\documentclass{amsart}
\numberwithin{equation}{section}
\allowdisplaybreaks[1]
\usepackage{amsmath, amssymb, amsgen, amsthm, amscd, xspace}

\newcommand{\Real}{\mathbb{R}}

\newcommand{\mk}{\mathbb{K}}
\newcommand{\mks}{\mathbb{K^*}}
\newcommand{\sig}{\sigma}

\newcommand{\del}{\delta}
\newcommand{\dd}{\partial}

\newcommand{\bfe}{\mathbf{E}}
\newcommand{\rme}{\mathrm{E}}
\newcommand{\bff}{\mathbf{F}}

\newcommand{\rmi}{\mathrm{I}}
\newcommand{\rmj}{\mathrm{J}}
\newcommand{\mat}{\mathrm{M}}
\newcommand{\mnk}{{\mat}_{n,k}}
\newcommand{\mkn}{{\mat}_{k,n}}
\newcommand{\mNk}{{\mat}_{N,k}}
\newcommand{\mnks}{{\mat}^{*}_{n,k}}
\newcommand{\rmx}{\mathrm{X}}
\newcommand{\rmd}{\mathrm{D}}
\newcommand{\rmp}{\mathrm{P}}
\newcommand{\rmps}{{\rmp}^{*}}
\newcommand{\rmpt}{{\rmp}^{t}}
\newcommand{\rmpst}{{\rmps}^{t}}
\newcommand{\half}{{\frac12}}

\newcommand{\cl}{\overline}
\newcommand{\cf}{\emph{cf}}
\newcommand{\ie}{\emph{ie}}
\newcommand{\opcit}{\emph{op cit}}
\newcommand{\gl}{{\mathfrak{gl}}}
\newcommand{\gln}{{\gl_n}}
\newcommand{\glN}{{\gl_N}}
\newcommand{\glk}{{\gl_k}}

\newcommand{\go}{{\mathfrak{o}}}
\newcommand{\oN}{{\go_N}}
\newcommand{\spg}{{\mathfrak{sp}}}
\newcommand{\spn}{{\spg_{2n}}}
\newcommand{\spk}{{\spg_{2k}}}

\renewcommand{\gg}{{\mathfrak{g}}}
\newcommand{\gd}{{\gg}^{d}}
\newcommand{\Gb}{\bar{G}}
\newcommand{\gb}{\bar{\gg}}
\newcommand{\ggs}{\gg^*}
\newcommand{\gp}{{\gg'}}
\newcommand{\gps}{\gp^*}
\newcommand{\zg}{\mathcal{Z}(\gg)}

\newcommand{\zgn}{\mathcal{Z}(\gln)}
\newcommand{\zgk}{\mathcal{Z}(\glk)}
\newcommand{\ugn}{U(\gln)}
\newcommand{\ugk}{U(\glk)}

\newcommand{\ug}{U(\gg)}
\newcommand{\ugp}{U(\gp)}
\newcommand{\usp}{U(\spg(W))}
\newcommand{\uoN}{U(\oN)}
\newcommand{\Ws}{W^*}
\newcommand{\Wd}{{W}^{d}}
\newcommand{\sg}{S(\gg)}
\newcommand{\sgp}{S(\gp)}
\newcommand{\SW}{S(W)}
\newcommand{\SnW}{S^n(W)}
\newcommand{\StW}{S^2(W)}
\newcommand{\StWs}{S^2(\Ws)}
\newcommand{\SU}{S(U)}
\newcommand{\SV}{S(V)}
\newcommand{\StU}{S^2(U)}
\newcommand{\StV}{S^2(V)}
\newcommand{\StUV}{S^2(U\otimes V)}
\newcommand{\SUpV}{S(U\oplus V)}
\newcommand{\LtU}{\Lambda^2(U)}
\newcommand{\LtV}{\Lambda^2(V)}

\newcommand{\Stm}{S^2(\mnk)}
\newcommand{\Stms}{S^2(\mnk^*)}
\newcommand{\Harm}{\mathcal{H}}
\newcommand{\OO}{\mathcal{O}}
\newcommand{\OC}{{\cl\OO}}
\newcommand{\I}{\mathcal{I}}
\newcommand{\A}{\mathcal{A}}
\newcommand{\weyl}{\mathcal{W}}
\newcommand{\weyld}{\weyl^d}
\newcommand{\tr}{\operatorname{tr}}
\newcommand{\sgn}{\operatorname{sgn}}
\newcommand{\ad}{\operatorname{ad}}
\newcommand{\Ad}{\operatorname{Ad}}
\newcommand{\ann}{\operatorname{Ann}}
\newcommand{\gr}{\operatorname{gr}}

\newcommand{\End}{\operatorname{End}}
\newtheorem{theorem}[equation]{Theorem}
\newtheorem{cor}[equation]{Corollary}
\newtheorem{prop}[equation]{Proposition}

\theoremstyle{remark}
\newtheorem{remark}[equation]{Remark}

\newtheorem{question}[equation]{Question}
\theoremstyle{definition}
\newtheorem{definition}[equation]{Definition}

\begin{document}
%
%
\title[{T}ransfer of ideals and quantization]
{Transfer of ideals and quantization of small nilpotent orbits}
\author{Victor Protsak}
\email{protsak@\,math.cornell.edu}
\address{Department of Mathematics, Cornell University, Ithaca, NY 14853}
%
\date{February 13, 2008}
%
%
\begin{abstract}
We introduce and study a transfer map between ideals of the 
universal enveloping
algebras of two members of a reductive dual pair of Lie algebras. Its
definition is motivated by the approach to the real Howe duality 
through the theory of Capelli identities. We prove that this map
provides a lower bound on the annihilators of theta lifts of representations
with a fixed annihilator ideal. We also show that in the algebraic
stable range, transfer respects the class of quantizations of 
nilpotent orbit closures.  As an
application, we explicitly describe quantizations of small nilpotent orbits
of general linear and orthogonal Lie algebras and give presentations of 
certain rings of algebraic differential operators.
We consider two algebraic versions of Howe duality and
reformulate our results in terms of noncommutative Capelli identities.
\end{abstract}
\keywords{{U}niversal enveloping algebra, primitive ideal, nilpotent orbit,
quantization, Capelli identity, reductive dual pair, theta correspondence,
algebraic Howe duality}
\subjclass[2000]{Primary 17B35; Secondary 22E46}
%
%
\maketitle
%
%
\section{Introduction and the main results}
%
%
There is a well known relation, going back to Alfredo Capelli, between
the centers of the universal enveloping algebras of two Lie algebras
forming a reductive dual pair in the sense of Howe, 
\cite{Howe_Remarks, Howe_Transcending, Przebinda_CapelliHC}. 
From the modern viewpoint, it arises by restricting the oscillator 
representation of the real symplectic group to the members of the dual pair 
and manifests itself via the so-called Capelli identities, 
\cite{Howe_Umeda, Itoh_Advances, Molev_Nazarov}. 
The case originally considered by Capelli  in \cite{Capelli_O}
is now known as the {$(GL_n, GL_m)$ duality}, \cite{Howe_Schur}. 
Capelli identities have important applications to the study of the 
theta correspondence: they determine the correspondence 
of the infinitesimal characters of the representations matched 
under the Howe duality, \cf \cite{Howe_Theta, Przebinda_Infinitesimal}.
The goal of the present paper is to extend this formalism 
from central elements 
to more general $\ad$-invariant subspaces and ideals 
of universal enveloping algebras. We define a certain \emph{transfer map} 
between two-sided ideals of the universal enveloping
algebras of the two members of a reductive dual pair, establish its basic
properties, and compute 
it in a special case, for the annihilator of
a one-dimensional representation of $\gp$, i.e.~a codimension one ideal of
$\ugp$. If
the reductive dual pair is in the stable range with $\gp$ the smaller
member, this amounts to a noncommutative analogue of 
the Kraft--Procesi lifting of the zero orbit, \cite{KP2}, and yields 
an ideal $I$ of $\ug$ quantizing the closure of a small nilpotent orbit 
in a classical Lie algebra $\gg$. We explicitly describe an 
$\ad(\gg)$-invariant subspace generating $I$ when $\gg$ is a general 
linear or orthogonal Lie algebra. 

Let $\mk$ be a field of characteristic zero, $W$ a symplectic vector space 
over $\mk$, $\spg(W)$ the symplectic Lie algebra of $W$ and
$(\gg,\gp)$ an irreducible reductive dual pair of Lie algebras. 
By definition, this means that $\gg$ and $\gp$ are Lie subalgebras
of $\spg(W)$ that are full mutual centralizers, $W$ is completely
reducible as a $\gg$-module and a $\gp$-module, 
and the action of 
$\gg\oplus\gp$ on $W$ is symplectically irreducible. Reductive dual
pairs were classified by Howe, see \cite{Howe_Theta}. If $\mk$ is algebraically
closed then either (1) both $\gg$ and $\gp$ are general linear Lie algebras
(type II), or (2) one of them is a symplectic and the other is an 
orthogonal Lie algebra (type I).
The embedding of the pair $(\gg,\gp)$ into $\spg(W)$ is determined 
uniquely modulo symplectic automorphisms of $W$. 
(Later on we will fix a particular realization of the dual pair.) Consider
the Weyl algebra of $W$, denoted $\weyl$. 
Its definition and properties are recalled in Section \ref{subsec:weyl}, but 
the salient point is that the universal enveloping algebras of 
$\gg$ and $\gp$ have canonical homomorphisms into $\weyl$, 
and it follows from Howe's Double Commutant Theorem, \cite{Howe_Remarks}, that 
the images of these homomorphisms are the full mutual centralizers in 
the even Weyl subalgebra.
\begin{definition}
Let $L:\ugp\to\weyl$ and $R:\ug\to\weyl$ be the canonical homomorphisms
into the Weyl algebra. For any subspace $I'$ of $\ugp$, its
\emph{transfer} $\theta(I')$ is defined to be the subspace of $\ug$ which is 
the pull-back via $R$ of the right ideal generated by the image of $I'$
in $\weyl$:
\begin{equation}
\theta(I') = R^{-1}(L(I')\cdot\weyl).
\end{equation}
\end{definition}
The preimage $R^{-1}$ should be understood in the following way: we 
consider those elements of $L(I')\cdot\weyl$ which belong to $R(\ug)$, or
equivalently (by the Double Commutant Theorem), 
which are $G'$-invariant, and lift them
to $\ug$. The lifting is well-defined on the level of subspaces and
the resulting subspace of $\ug$ contains $\ker R$.
It is easy to see that $I=\theta(I')$ is always a right ideal in $\ug$ and
that if $I'$ is a two-sided ideal of $\ugp$ then $I$ is a two-sided ideal
of $\ug$. The next theorem motivates this definition: it
shows that the transfer map is compatible with
the Howe duality correspondence for real Lie groups.
\begin{theorem}\label{thm:bound}
Let $(G_0, G'_0)$ be an irreducible reductive dual pair
of real Lie groups with the complexified Lie algebras $\gg$ and $\gp$.
Suppose that $\rho$ and $\rho'$ are irreducible admissible 
representations of the metaplectic covers of $G_0$ and $G'_0$ 
that correspond to each other under the real Howe duality correspondence, 
\cite{Howe_Transcending}. Denote by $I$ and $I'$ the annihilators
of $\rho$ and $\rho'$ in the universal enveloping algebras $\ug$
and $\ugp$. Then $\theta(I')\subseteq I$. 
\end{theorem}
Generalizing Howe and Umeda's 
treatment of the classical Capelli identities, in Section \ref{sec:duality} 
we interpret the ideal $\theta(I')$ as a system of noncommutative
Capelli identities.
Note that $\theta(I')$ provides only a lower bound for the annihilators
of the representations of the metaplectic cover of $G_0$ that 
are matched under the Howe 
duality correspondence with representations
of the metaplectic cover of $G'_0$ with a fixed primitive annihilator $I'$. 
It is a subtle problem to formulate
sufficient conditions for the equality to hold. This question
is further discussed at the end of the paper.

The transfer in general is rather complicated. However, we show that
ideals of special type, \emph{quantizations of nilpotent 
orbit closures}, behave well under the transfer.
\begin{definition}
An ideal $I$ of the universal enveloping algebra $\ug$ is a
quantization of the orbit closure $\OC$ if the associated
graded ideal $\gr I$ is the prime ideal defining $\OC$.
\end{definition}
The property of being a quantization of an orbit closure is very 
strong and it implies that the ideal is completely prime and primitive.
When $\gg$ is general or special linear Lie algebra, the converse statement 
had been conjectured by Dixmier and was proved by Moeglin, \cite{Moeglin_CP}; 
it fails for other types.
\begin{theorem}\label{thm:quantize}
Let $(\gg,\gp)$ be an irreducible reductive dual pair of Lie algebras
in the stable range with the smaller member $\gp$. Suppose that $I'$
is a completely prime, primitive ideal of $\ugp$ quantizing an 
orbit closure $\OC'$. Then its transfer $\theta(I')$ is
a completely prime, primitive ideal of $\ug$ quantizing an orbit
closure $\OC$, the Kraft--Procesi lifting of $\OC'$. 
\end{theorem}
We give a short direct proof of this theorem based on the 
flatness of the moment map $l$ in
the stable range. 

The preceding theory is made more explicit in the case of
a one-dimen\-si\-o\-nal representation of $\gp$
and its annihilator $I'$. This ideal has codimension one in $\ugp$,
quantizes the zero orbit, and has an obvious system
of generators. We explicitly compute an $\ad$-invariant
system of generators of its transfer $I=\theta(I')\subseteq\ug$. 
By Theorem \ref{thm:quantize}, the ideal $I$ obtained in this way is a 
quantization of the prime ideal 
defining an orbit whose Young diagram has exactly two columns (we 
call such orbits \emph{small}). 
It is convenient to state our description of the generators of $I$
in terms of ``noncommutative linear algebra'', using a certain matrix 
with elements in $\ug$ (see Section \ref{sec:transfer} for the precise 
formulations, including the notation).

In the general linear case, $\gg=\gln$, let us denote by $\bfe$ the 
$n\times n$ matrix over $\ugn$ whose
$ij$th entry is $\rme_{ij}$, the standard generator of $\gln$ given by the
$(i,j)$ matrix unit. In Theorem \ref{thm:transfer_gl} we show that 
the ideal $I$ is generated by the trace of $\bfe$ shifted by
a certain constant, the matrix entries of an explicitly given
quadratic polynomial in $\bfe$, and \emph{quantum minors} of $\bfe$
introduced in \cite{Protsak_RankCapelli}, \cf Definition \ref{def:qminor}.
The main result in the symplectic--orthogonal case, Theorem 
\ref{thm:transfer_spo}, has similar form, with 
the \emph{quantum pfaffians} of \cite{Itoh_Advances, Molev_Nazarov}
replacing the quantum minors, \cf Definition \ref{def:pfaff}.
Note that, in contrast with the case of minimal nilpotent orbits
considered by Kostant, for a more general small nilpotent orbit
there are relations of degree greater than two arising from 
quantum minors and pfaffians.

It follows from the work of Borho--Brylinski, \cite{BB1}, 
and Levasseur--Stafford, \cite{LS}, that
the ring of algebraic differential operators on the 
(projective, nonsingular) Grassmann variety or the (affine, singular) rank
variety may be identified with the quotient of $\ug$ by the ideal of 
this type. Therefore, we obtain presentations of these noncommutative 
rings of differential operators by generators and relations. 
%
%
%
%
\section{Preliminaries}\label{sec:prelim}
%
%
The goal of this section is to set up the terminology and review 
the basic properties of Weyl algebras, reductive dual pairs, and 
nilpotent orbits. We present a formulation of the theory of 
spherical harmonics in the form needed to compute transfers of ideals. 
Sections \ref{subsec:weyl} and \ref{subsec:harmonics}
follow the approach of Howe's Schur Lectures \cite{Howe_Schur}, where
many details may be found.
%
%
\subsection{Weyl algebra}\label{subsec:weyl}
%
%
Let $W$ be a symplectic vector space. The Weyl algebra $\weyl$ 
is a filtered algebra generated by
the vector space $W$ in degree 1, modulo the relations
$w_1 w_2 - w_2 w_1 = \langle w_1, w_2 \rangle$.
Its associated graded algebra is the symmetric algebra of $W$, the 
group $Sp(W)$ acts on $\weyl$ by filtered algebra automorphisms, and $\weyl$ 
decomposes into a direct sum over non-negative integers $n$
of finite-dimensional $Sp(W)$-modules $\weyl^n\simeq\SnW$.
The components of this decomposition
are irreducible and mutually non-isomorphic modules, 
therefore, there is a canonical $Sp(W)$-invariant
linear symbol map between $\weyl$ and $\SW$. On the elements of degree
$2$ in $\SW$, the (inverse) symbol map is given by the formula:
\begin{equation*}
s^{-1}(w_1 w_2)=\half(w_1w_2 + w_2w_1).
\end{equation*}
Furthermore, 
$[\weyl^2,\weyl^n]\subseteq\weyl^n$ and
the adjoint action of $\weyl^2$ on $\weyl$
coincides with the infinitesimal action of $\spg(W)$,
the symplectic Lie algebra of $W$. 
The direct sum $\A$ of all even components $\weyl^{2k}, k\geq 0$, is
a subalgebra of $\weyl$ called the \emph{even Weyl algebra}. It is generated
by $\weyl^2$, the lowest positive degree component, and hence $\A$ is 
a quotient of the universal enveloping algebra $\usp$. Howe's Double Commutant
Theorem states that, given a reductive dual pair of Lie algebras, 
the subalgebra $R(\ug)$ of the Weyl algebra coincides with the subalgebra of 
$G'$-invariants in $\weyl$, which also coincides with the space of 
$\ad(\gp)$-invariants in the even Weyl algebra, \ie with the commutant of
$\gp$ or $L(\ugp)$ in $\A$.
%
%
\subsection{Moment maps and the Kraft--Procesi lifting}
\label{subsec:KP}
%
%
The Weyl algebra $\weyl$ is a noncommutative filtered algebra, and
the maps $L:\ugp\to\weyl$ and $R:\ug\to\weyl$ become 
filtered maps if the elements from $\gp\subset\ugp$ and $\gg\subset\ug$ 
are assigned degree two. Taking the associated graded maps, we get a pair
of commutative algebra homomorphisms, $\gr L: \sgp\to\SW$ and 
$\gr R: \sg\to\SW$, and we denote the corresponding morphisms of
algebraic varieties by $l:\Ws\to\gps$ and $r:\Ws\to\ggs$ and
refer to them as the moment maps. 

A reductive dual pair $(\gg,\gp)$ of Lie algebras canonically determines
a reductive dual pair $(G, G')$ of algebraic groups over $\mk$
as follows: $G$ is the
pointwise stabilizer of $\gp$ in $Sp(W)$ and $G'$ is the pointwise stabilizer
of $\gg$ in $Sp(W)$. If $\gg$ is a general linear Lie algebra then $G$ is
the corresponding general linear group, and if $\gg$ is the Lie algebra
of infinitesimal symmetries of a bilinear or a sesquilinear form then
$G$ is the full isometry group of the form 
($G$ may be disconnected), and similarly for $G'$.
We call $G$ and $G'$ the classical groups of 
$\gg$ and $\gp$, they form a reductive dual pair in $Sp(W)$.
The maps $l$ and $r$ are equivariant with respect to 
the natural actions of $G$ and $G'$ on $\ggs, \gps,$ and $\Ws$.
Moreover, the First Fundamental Theorem of Classical Invariant 
Theory for vector representations states that 
$l$ (respectively, $r$) 
is the affine factorization map for the action of the group $G$ 
(respectively, $G'$) on $W$ in the sense of invariant theory.

Nilpotent orbits of a classical Lie group $G$ 
over an algebraically closed field of characteristic zero with $N$-dimensional
defining module are parametrized by
certain integer partitions of $N$ whose parts are the sizes of the Jordan
blocks of any nilpotent matrix in the orbit. These partitions must satisfy 
additional parity conditions in the symplectic and orthogonal cases,
\cite{ColMc}.
We call an orbit $\OO$ \emph{small} if the largest part of the
corresponding partition is equal to two, equivalently, if its Young diagram
consists of two columns. Small nilpotent orbits are classified 
by their \emph{rank}, the rank of any matrix in the orbit. The rank is
an integer between $1$ and $N/2$ equal to the length of the 
second column of the Young diagram, and must be even if 
$\gg=\oN$ is an orthogonal Lie algebra. (In the case $G=O_N, N=4k$, the small
nilpotent $G$-orbit of rank $2k$ is \emph{very even}, and breaks up into two 
nilpotent orbits for the special orthogonal group $SO_N$.)
For example, the minimal orbit is small. Small orbits are spherical and
form a complete chain in the poset of nilpotent orbits ordered by the
inclusions of their closures.
\begin{definition}
Let $\OO'$ be a nilpotent coadjoint orbit in $\gps$ with Zariski closure
$\OC'$. The Kraft--Procesi lifting associates with it 
the $G$-invariant affine subvariety $\OC=r(l^{-1}(\OC'))$ of $\ggs$.
\end{definition}
The variety $\OC$ arising from the Kraft--Procesi lifting
is the closure of a nilpotent coadjoint
orbit $\OO\subseteq\ggs$. Moreover, if $(\gg, \gp)$ is in the stable range 
with $\gp$ the smaller member then the Young diagram of $\OO$ is obtained
from the Young diagram of $\OO'$ by adding an extra first column of the
appropriate length. In particular, if $\OO'$ is the zero orbit then 
$\OO$ is small.
For $(\gg, \gp)$ in the stable range, these facts were proved by 
Kraft and Procesi, and this is precisely the case we are going to use.
More general results appeared in \cite{DKP}.  
%
%
\subsection{Reductive dual pairs and the theory of harmonics}
\label{subsec:harmonics}
%
%
Let $(\gg,\gp)$ be a reductive dual pair of Lie algebras.
We say that $(\gg,\gp)$ is in the stable range, with $\gp$ 
the smaller member, if the dimension of the defining module of 
$\gg$ is at least twice the dimension of the defining module 
of $\gp$. Thus $(\gln, \glk)$ is in the stable range with $\glk$ the
smaller member if and only if $n\geq 2k$. More generally, we consider
reductive dual pairs with $\gp$ the smaller member, defined through
comparison of the dimensions of the defining modules.

The \emph{double} of the symplectic vector space $W$ is the vector
space $\Wd = W\oplus W^*$, with a natural symplectic form making this
direct sum decomposition into a lagrangian polarization. 
We denote by $\gd$ the centralizer of $\gp$ in the symplectic Lie algebra 
of $\Wd$ and call it the \emph{double} of the Lie algebra $\gg$. 
Then $(\gd,\gp)$ is a reductive dual pair in $\spg(\Wd)$. Let $\gb$ be
the centralizer of $\gg$ in $\gl(W)$, the general linear algebra of $W$.
The Weyl algebra $\weyld$
of the double $\Wd$ acts on $\SW$ by polynomial coefficient differential
operators in the standard way. The Lie algebra
$\gd$ has the following decomposition arising from the polarization 
of $\Wd$ and the induced decomposition of the second component of $\weyld$ as 
$\StW\oplus(W\otimes\Ws)\oplus \StWs$:
\begin{equation}\label{eq:decomp}
\gd=\gg\oplus\gb\oplus\gg^{*},
\end{equation}
where the three summands correspond to $(2,0)$, $(1,1)$, and $(0,2)$ 
components of $\gd$ in the terminology of \cite{Howe_Transcending}. In the
setting of an irreducible reductive dual pair $(\gg,\gp)$ over an 
algebraically closed field, we have the following possibilities: 
(1) $\gg=\gln$ is a general linear
algebra and $\gb=\gln\times\gln$ with $\gg$ embedded diagonally; 
(2) $\gg=\spn$ is a symplectic Lie algebra and $\gb=\gl_{2n}$ is the 
corresponding general linear algebra; (3) $\gg=\oN$ is an orthogonal Lie 
algebra and $\gb=\glN$ is the corresponding general linear Lie algebra.

The space $\Harm$ of $G'$-harmonics is the subspace of $\SW$ annihilated by
the action of the $G'$-invariant constant coefficient differential operators.
Equivalently, it is the kernel for the action of the 
$(0,2)$-component of $\gd$ in the decomposition \eqref{eq:decomp}. The space
of harmonics is $G'$-invariant and has a direct sum decomposition into 
isotypic components corresponding to finite-dimensional simple $G'$-modules.
It was discovered by Gelbart and generalized by Howe that 
harmonics have very special structure:
the space $\Harm_{\tau}$ of 
$G'$-harmonics of type $\tau$ is an irreducible 
$G'\times\Gb$-module of type $\tau\otimes\sigma$ and generates 
the $\tau$-isotypic component
of $\SW$. Here $\sigma$ is an isomorphism class of simple
finite-dimensional $\Gb$-modules and is uniquely
determined by $\tau$. The correspondence between $\sigma$ and $\tau$
is explicitly described in \cite{Howe_Schur}.

Denote the associated graded algebras of 
$\weyl$ and $\ugp$ by 
$A$ and $B$. Since $\gp$ is the smaller member of the dual pair,
the map $\gr L$ is an injection, and we will pretend that it's just the
identity map, so that $B$ is a subalgebra of $A$ and $\ugp$ is a subalgebra
of $\weyl$. 
The next theorem is a restatement of the decomposition of $A=\SW$ into
invariants and harmonics arising from the reductive dual pair $(\gd,\gp)$ in
$\spg(\Wd)$.
\begin{theorem}\label{thm:generate}
Suppose that $\tau$ is an irreducible $\gp$-module, $\tau^*$ its 
contragredient, and $V_\tau$ is an
$\ad(\gp)$-invariant subspace of $B=\sgp$ isomorphic to $\tau$.
Let $\sigma$ be the $\gd$-module 
corresponding to $\tau^*$ in the space of harmonics. 
Then $(V_\tau A)^{G'}$ is generated 
as a $\sg$-module by the subspace $U_\sigma$, the result
of the $Ad(G')$-invariant multiplication
pairing between $V_\tau$ and $\Harm_{\tau^*}$, and $U_\sigma$ is an
irreducible $\bar{G}$-module of type $\sigma$.
\end{theorem}
%
%
%
%
%
\section{Quantizations of nilpotent orbit closures}\label{sec:quantize}
%
The main goal of this section is to establish  Theorem \ref{thm:quantize}.
Our proof is based on a pair of remarkable
properties of the moment map $l:\Ws\to\gps$ in the stable range.
\begin{prop}\label{prop:flat}
Let $(\gg, \gp)$ be an irreducible reductive dual pair in stable 
range with $\gp$ the smaller member. Then the moment map 
$l:\Ws\to\gps$ is flat and its geometric fibers are irreducible.
In fact, $\SW$ is a free $\sgp$-module.
\end{prop}
\begin{proof}
By the method of associated cones and the dimensional criterion of
flatness, it is enough to verify that the zero fiber $l^{-1}(0)$ is
reduced and irreducible, of dimension $\dim\Ws - \dim\gps$. By 
\cite{KP1, KP2}, the zero
fiber is an irreducible complete intersection. The last assertion
follows because $\SW$ is a graded algebra and $\sgp$ is its graded
subalgebra.
\end{proof}
It is a general fact from commutative algebra that 
in the situation of Proposition \ref{prop:flat},
the scheme-theoretic preimage
of an irreducible subscheme of the base is reduced and irreducible.
\begin{prop}\label{prop:prime}
Let $A$ be an affine algebra over a field $\mk$, $B$ 
its affine $\mk$-subalgebra. 
Assume that $A$ is flat as a $B$-module and
for any maximal ideal
$m$ of $B$, the ideal $mA\subseteq A$ is prime. Then for any prime ideal
$P$ of $B$, the ideal $PA$ of $A$ is prime.
\end{prop}
\begin{proof}
In an affine algebra, a prime ideal is the intersection of the maximal
ideals that contain it. The flatness of $A$ over
$B$ implies that intersections are compatible with the ``base change''
from ideals of $B$ to ideals of $A$,
$S\mapsto SA$. 
Therefore, $PA$ is the intersection of the ideals
of the form $mA$ for various maximal $m\supseteq P$, which are prime
by assumption. 
Suppose that $fg\in PA$, then for every such $m$, 
either $f$ or $g$ is contained 
in $mA$. Since these are Zariski closed conditions on $m$ and $V(P)$ is 
irreducible, either $f$ or $g$ is contained in $mA$ for all $m$
containing $P$. It follows that either $f$ or $g$ is contained in $PA$, 
establishing its primeness.
\end{proof}
\begin{proof}[Proof of Theorem \ref{thm:quantize}]
We show that $I$ has the desired property using 
Propositions \ref{prop:flat} and 
\ref{prop:prime} and the standard fil\-te\-red--gra\-ded techniques.
As in Section \ref{subsec:harmonics}, 
denote $\gr\weyl$ and $\gr\ugp$ by $A$ and $B$ and view $B$ as a 
subalgebra of $A$.  Then $\ugp$ is a subalgebra
of $\weyl$. Moreover, by Proposition \ref{prop:flat}, $A$ is a flat $B$-module.
Here is the crucial step: the flatness of $A$ over $B$ implies by
\cite{Bjork}, Chapter 2, Theorem 8.6 that 
\begin{equation}
\gr(I'\weyl)=(\gr I')A.
\end{equation}
By assumption, the ideal $P=\gr I'$ is prime. 
Proposition \ref{prop:prime} then shows that the ideal $PA$ is prime, and the
corresponding irreducible subvariety of $\Ws$ is $l^{-1}(\OC')$.
Pulling back $I'\weyl$ via $R$ and recalling that $R(\ug)$ coincides
that with the $G'$-invariants of $\weyl$, we see that 
$\gr\theta(I')=(\gr R)^{-1}(PA)$ and that the ideal in the right
hand side defines the subvariety $r(l^{-1}(\OC'))=\OC.$ Thus
$\theta(I')$ is a quantization of $\OC$, as we claimed.
\end{proof}
Using flatness to prove good properties of
the associated graded construction is a staple of noncommutative
algebra and goes back at least to Bj{\"o}rk's book \cite{Bjork}.
The papers \cite{Holland_ENS,Schwarz_ENS} used a more
sophisticated approach via the Koszul complex, which resulted in
the analogue of Theorem \ref{thm:quantize} in a 
related context of quiver representations, but only for  
ideals $I'$ of special form, the annihilators of one-dimensional 
representations.
By providing a more direct proof, we are
able to treat the case of arbitrary $I'$.
\begin{remark}
Our argument can be adapted to analysis of the effect of transfer
on the associated cycles of ideals. 
We hope to return to this question in a future
publication. 
\end{remark}
%
%
%

%
\section{Transfer of codimension one ideals}\label{sec:transfer}
%
In this section we find generators for the transfer of a codimension one 
ideal $I'\subset\ugp$ in the stable range.
First, by an explicit computation, we will exhibit a certain 
$\ad$-invariant finite-dimensional subspace of $\ug$ contained 
in the ideal $I=\theta(I')$. Then we use Theorem \ref{thm:generate}
to conclude that this subspace generates $I$ as a (left or right) 
$\ug$-module. This step also applies to transfers of more 
complicated ideals of $\ugp$.
\subsection{General linear case}
The Weyl algebra associated with the reductive dual
pair $(\gln,\glk)$ may be identified with the algebra of polynomial
coefficient differential operators on $n\times k$ matrices, corresponding
to the realization $W=\mnk\oplus\mnks$ for the symplectic vector space.
Let $\rmx$ be the $n\times k$ matrix
with entries $x_{ia}$, the coordinate functions on $\mat_{n,k}$ 
and $\rmd$ be the $n\times k$ matrix with entries $\dd_{jb}$,
the corresponding partial derivatives. 
Denote the standard generators of $\gln$ given by the matrix units
by $\rme_{ij}, 1\leq i,j \leq n$, and arrange them into a single
$n\times n$ matrix over $\ugn$:
\begin{equation}
\bfe=\begin{bmatrix}
\rme_{11} & \rme_{12} & \ldots & \rme_{1n} \\
\rme_{21} & \rme_{22} & \ldots & \rme_{2n} \\
\ldots & \ldots & \ldots & \ldots \\
\rme_{n1} & \rme_{n2} & \ldots & \rme_{nn} 
\end{bmatrix}\in\mat_n(\ugn).
\end{equation}
It turns out that the image of $\bfe$ under the homomorphism
$R$ into $\weyl$ can be (almost) factorized as follows:
$R(\bfe)=\rmx \rmd^t+\frac{k}{2}I_n$. Similarly, if we denote by
$\bfe'$ the $k\times k$ matrix whose entries are the standard generators
of $\glk$ then $L(\bfe')=\rmx^t \rmd+\frac{n}{2}I_k$. In order
to simplify some of the formulas that follow, we now modify the definition
of $L$ and $R$, so that 
\begin{equation}\label{eq:gen_gl}
L(\bfe')=\rmx^t \rmd, \quad R(\bfe)=\rmx \rmd^t.
\end{equation}
These \emph{unnormalized} homomorphisms $L$ and $R$ have almost all 
the properties of the original ones, except that the images
of $\gln$ and $\glk$ no longer belong to $\weyl^2$, the symplectic
subalgebra $\spg_{2kn}$ of the Weyl algebra. We will, therefore, 
consider the ``unnormalized transfer'' based on the modified maps 
$L$ and $R$ defined by the formula \ref{eq:gen_gl}.  
\begin{definition}\label{def:qminor}
Let $I=(i_1,\ldots,i_{k+1}), J=(j_1,\ldots,j_{k+1})$ be two
$(k+1)$-element sequences of indices from $1$ to $n$.
The \emph{quantum minor} $\bfe_{IJ}$ is an element of $\ugn$ 
given by the following noncommutative determinants (of either row or 
column type):
\begin{equation*}
\begin{aligned}
\bfe_{IJ}
&=\sum_{\sig}\sgn(\sig)(\bfe+k)_{i_{\sig(1)}j_1}\ldots
(\bfe+1)_{i_{\sig(k)}j_{k}}\bfe_{i_{\sig(k+1)}j_{k+1}}\\
&=\sum_{\sig}\sgn(\sig)\bfe_{i_1 j_{\sig(1)}}(\bfe+1)_{i_2 j_{\sig(2)}}
\ldots(\bfe+k)_{i_{k+1}j_{\sig(k+1)}},
\end{aligned}
\end{equation*}
where scalars denote the corresponding scalar multiples of
$\rmi_n$ and the sum is taken over all permutations $\sig\in S_{k+1}$.
\end{definition}
Quantum minors of given order quantize ordinary minors of 
$n\times n$ matrices and span an $\ad$-invariant subspace of 
$\ugn$. Moreover, we have proved in \cite{Protsak_RankCapelli} 
that the quantum minors of order $k+1$ strictly generate the
ideal $\ker R$, called the \emph{rank ideal} associated with the reductive
dual pair $(\gln, \glk)$.

It follows immediately from the definition of transfer that
the transfer of any ideal $I'$ of $\ugp$
contains $\ker R$. Conversely, if we find a set of
generators for $L(I')\weyl\cap R(\ug)$ as an $R(\ug)$-module 
then supplanting them with the quantum minors of order $k+1$, 
we obtain a set of generators of $\theta(I')$ as a $\ug$-module.

\begin{prop}
The second component $\weyl^2$ of the Weyl algebra 
decomposes under the adjoint $GL_k\times GL_n$-action as follows:
\begin{equation}\label{eq:decomp_gl}
\weyl^2\simeq\StW=\Stm\oplus(\mnk\otimes\mnks)\oplus\Stms.
\end{equation}
Moreover, the middle summand is isomorphic to $\gln\otimes\glk$. 
\end{prop}
\begin{proof}
The first isomorphism is given by the $GL(W)$-invariant unnormalized
symbol map. The decomposition of $\StW$ follows from the multiplicativity
of the symmetric power, $\SUpV=\SU\otimes\SV$. The last assertion holds
due to the canonical isomorphisms $\mnk\simeq\mk^n\otimes\mks^k$, 
$\mnks\simeq\mks^n\otimes\mk^k$, $\gln\simeq\mk^n\otimes\mks^n$, 
$\glk\simeq\mk^k\otimes\mks^k$.
\end{proof}
Let $V'$ be the subspace of $\ugk$ spanned by the entries of the matrix
$\bfe'+t\rmi_k$, where $t$ is a fixed scalar, $I'$ be the ideal of 
$\ugk$ generated $V'$, and let $I=\theta(I')$ be its transfer to $\ugn$.
It is clear that $V'$ is isomorphic to $\glk$ as an
$\ad(\glk)$-module and that the element $\tr\bfe'+kt$ spans the 
one-dimensional $\ad$-invariant (central) subspace of $V'$. Since 
\begin{equation*}
\tr L(\bfe')=\sum_{ia}x_{ia}\dd_{ia}=\tr R(\bfe),
\end{equation*}
we see that $\tr\bfe+kt$ is contained in $I$. 

To find further elements of the ideal $I$, let us
use the $\ad(\glk)$-invariant multiplication pairing of $V'$ with the
subspace of $\weyl^2$ corresponding under the symbol map
to the middle summand in the 
decomposition \eqref{eq:decomp_gl}. This subspace is
spanned by the elements $\{x_{ib}\dd_{ja}\}$ and we compute:
\begin{equation*}
\sum_{ab} L(\bfe'_{ab})x_{ib}\dd_{ja} =
\sum_{ab,l} x_{la}\dd_{lb}x_{ib}\dd_{ja}=
\sum_{ab,l} x_{ib}\dd_{lb}x_{la}\dd_{ja}+(k-n)\del_{ab}x_{ia}\dd_{ja},
\end{equation*}
where the summation is over $a,b$ between $1$ and $k$ and over $l$ between 
$1$ and $n$. The right hand side is equal to 
$\sum_l R(\rme_{il})R(\rme_{lj})+(k-n)R(\rme_{ij})$, whose preimage
under $R$ is the $ij$ entry of $\bfe^2+(k-n)\bfe$. Replacing 
$L(\bfe')$ with $L(\bfe')+t\rmi_k$ has the effect of adding 
$t\sum_a x_{ia}\dd_{ja}=tR(\rme_{ij})$ to both sides. Therefore,
the subspace $V$ of $\ugn$ spanned by $\tr\bfe+kt$ and  
the entries of the $n\times n$ matrix $\bfe^2+(k-n+t)\bfe$ is contained in $I$.
This subspace is $\ad(\gln)$-invariant 
and it follows from Theorem \ref{thm:generate} that
its image under $R$  generates  
$L(I')\weyl\cap R(\ug)$ as a $\ug$-module. 
We summarize the description of the transfer of a codimension one ideal 
of $\ugk$ in the following theorem.
\begin{theorem}\label{thm:transfer_gl}
Let $(\gg, \gp)=(\gln, \glk)$ with $2k\leq n$ and
let $I'$ be the annihilator of the one-dimensional module
$E'_{ab}\mapsto -t\del_{ab}$ over $\glk$. Then its unnormalized 
transfer $I$ is a quantization of the closure of the
small orbit $\OO_k$ of rank $k$ in $\gln$. Moreover,
\begin{enumerate}
\item[(A)]
The ideal $I$ has an $\ad$-invariant system of generators 
spanned by 
\begin{itemize}
\item
$\tr\bfe+kt$, 
\item
the entries of $p(\bfe)$, where $p(u)=u^2+(k-n+t)u$, 
\item
the quantum minors $\bfe_{IJ}$ of order $k+1$.
\end{itemize}
\item[(B)]
The corresponding primitive quotient of $\ugn$ may be identified
with the ring of differential operators on an algebraic variety in the
following cases:
for $t=0$, the quotient is isomorphic to the
algebra of global differential operators on the Grassmanian
$Gr(k,n)$; for an integer $t$, $k<t<n-k$, it is isomorphic to the 
algebra of the algebraic differential operators on the singular rank variety
formed by the $t\times (n-t)$ matrices of rank at most $k$. 
\end{enumerate}
\end{theorem}
\begin{proof}
Part (A) has already been established. To demonstrate Part (B), 
we will use Theorem
\ref{thm:bound} and Theorem \ref{thm:algtransfer}, whose proofs in Section 
\ref{sec:duality} are logically independent from the rest of the article.

Consider first the case of the rank variety. Levasseur and Stafford 
have studied in \cite{LS} the annihilator $J_k$ of the 
unitary highest weight module arising
as the unnormalized theta lift of the trivial representation for the
reductive dual pair $(U_{p,q},U_k)$ of compact type, where $p,q\geq k+1$
and proved that the corresponding primitive quotient is 
the ring of differential operators on the singular affine variety 
of $p\times q$ matrices of rank at most $k$. Let $p=t, q=n-t$, 
then the assumption on $p,q$ holds. By Theorem \ref{thm:bound}, 
$I\subseteq J_k$. However, both $I$ and $J_k$ have the same 
associated variety, namely, the closure of the small orbit of rank $k$. 
Therefore, these ideals coincide. 

The case of the Grassmannian is analogous.
Borho and Brylinski proved
in \cite{BB1} that the homomorphism from $\ugn$ into the
algebra of global differential operators on the
Grassmannian $Gr(k,n)$ induced by the $\gln$-action \eqref{eq:gen_gl}
on $n\times k$ matrices is surjective and its kernel $J$ quantizes
the closure of the small nilpotent orbit of rank $k$.
By Theorem \ref{thm:algtransfer}, 
$J$ contains $I$, therefore, these ideals coincide.
\end{proof}
With applications to theta correspondence in mind, Section \ref{sec:duality},
we formulate the description of the transfer resulting from using the 
\emph{normalized} maps $L$ and $R$ (defined before \eqref{eq:gen_gl}).
\begin{cor}\label{cor:transfer_gl}
Let $(\gg, \gp)=(\gln, \glk)$ with $2k\leq n$ and
let $I'$ be the annihilator of the one-dimensional module
$E'_{ab}\mapsto \alpha\del_{ab}$ over $\glk$. Then its normalized 
transfer $I\subset\ugn$ has an $\ad$-invariant generating set
spanned by the following elements:
\begin{itemize}
\item
$\tr\bfe-k\alpha$, 
\item
the entries of $p(\bfe)$, where $p(u)=(u-{k}/{2})(u-(n-k+\alpha)/{2})$, 
\item
the twisted quantum minors $\bfe_{IJ}(-{k}/{2})$ of order $k+1$.
\end{itemize}
\end{cor}
%
%
\subsection{Symplectic -- orthogonal case}
%
%
The arguments in this case are similar to the general linear case considered
above. 
The role of the matrix $\bfe$
is played by a skew-symmetric matrix $\bff$ whose entries are
the standard generators of the Lie algebra $\oN$ 
in its ``canonical'' realization
by the skew-symmetric $N\times N$ matrices.

Let us briefly recall Howe's description of the irreducible reductive
dual pairs of type I over an algebraically closed field.
Let $U$ be a symplectic and $V$ be an orthogonal vector space
over $\mk$, both finite-dimensional, then $W=U\otimes V$, endowed
with the product of the forms on $U$ and $V$, is a 
symplectic vector space. The isometry groups $Sp(U)$ and $O(V)$ of
the forms embed into $Sp(W)$ and form a reductive dual pair in it, 
and their Lie algebras $\spg(U)$ and $\go(V)$ form a 
reductive dual pair of Lie algebras in $\spg(W)$.
\begin{prop}
The second component $\weyl^2$ of the Weyl algebra 
decomposes under the adjoint $Sp(U)\times O(V)$-action as follows:
\begin{equation}\label{eq:decomp_spo}
\weyl^2\simeq\StW=\StU\otimes\StV\oplus\LtU\otimes\LtV.
\end{equation}
Moreover, $\StU\simeq\spg(U)$ and $\LtV\simeq\go(V)$ are 
the adjoint representations of $Sp(U)$ and $O(V)$, while $\LtU$ and $\StV$  
contain one-dimensional summands corresponding to the symplectic
form on $U$ and the orthogonal form on $V$.
\end{prop}
\begin{proof}
The first isomorphism is given by the $Sp(W)$-equivariant symbol
map. The decomposition of $\StUV$ as a $GL(U)\times GL(V)$ module
is given by the Cauchy formula, with the summands corresponding
to the partitions $(2)$ and $(1, 1)$ of $2$. The statements concerning
the adjoint representations are part of the theory of Cayley transform
for classical groups, and the last assertion is obvious.
\end{proof}
Let  $(\gg, \gp)=(\oN, \spk)$. We use an identification of the Weyl algebra
with the polynomial coefficient differential operators on the
space of $N\times k$ matrices and a 
particular realization of the reductive dual pair with the symplectic
vector space completely polarized. Similarly to the general linear case, 
consider the $N\times k$ matrices $\rmx$ and $\rmd$
consisting of the coordinate functions and the partial derivatives. 
Let $\rmp$ be the $N\times 2k$ matrix obtained by putting them together, and 
consider its transpose, its conjugate with respect to the symplectic form 
with the matrix $\rmj$, and the conjugate transpose, 
which are explicitly given in the block form
as follows:
\begin{equation*}
\rmp=\begin{bmatrix}
\rmx & \rmd
\end{bmatrix}, \quad 
\rmpt=\begin{bmatrix}
\rmx^{t} \\  \rmd^{t}
\end{bmatrix}, \quad 
\rmps=\begin{bmatrix}
\rmd^{t} \\  -\rmx^{t}
\end{bmatrix}, \quad 
\rmpst=\begin{bmatrix}
-\rmx  & \rmd
\end{bmatrix}, \quad 
J=\begin{bmatrix}
0 & -\rmi_k \\ \rmi_k & 0
\end{bmatrix}. 
\end{equation*}
The root generators of $\spk$ are arranged into a $2k\times 2k$ matrix
$\bff'$, and the standard generators of $\oN$ are arranged into an
$N\times N$ skew-symmetric matrix $\bff$. The images of $\bff'$ and 
$\bff$ under the homomorphisms $L$ and $R$ into the Weyl algebra are
the following matrices of differential operators on $\mNk$, 
\cf \cite{Itoh_Advances}:
\begin{equation}\label{eq:gen_spo}
L(\bff')=\rmp^t(\rmps)^{t}+\frac{N}{2}\rmi_{2k}, \quad
R(\bff)=\rmp\rmps+k\rmi_N.
\end{equation}
\begin{definition}\label{def:pfaff}
Let $I=(i_1,i_2,\ldots,i_{2k+2})$ be a sequence of indices from $1$ to
$N$. The \emph{quantum pfaffian} $Pf_I$ of order $k+1$ is the following
element of $\uoN$:
\begin{equation*}
Pf_I=\frac{1}{2^{k+1}(k+1)!}\sum 
\sgn(\sig)\bff_{i_{\sig(1)}i_{\sig(2)}}\ldots
\bff_{i_{\sig(2k+1)}i_{\sig(2k+2)}},
\end{equation*}
where the sum is taken over the permutations $\sig\in S_{2k+2}$. 
\end{definition}
Quantum pfaffians are
quantizations of the ordinary pfaffians of skew-symmetric $N\times N$ matrices.
They are skew-symmetric in their indices and span an $\ad$-invariant subspace 
of $\uoN$, \cite{Itoh_Advances, Molev_Nazarov}.
Combining the techniques of \cite{Protsak_RankCapelli} with
the identities proved in \opcit, one can 
demonstrate that the rank ideal
$\ker R$ is generated by the quantum pfaffians of order $k+1$.
\begin{prop}
Let $p(t)=t^2-(N/2-1)t$. Then the matrix $p(\bff)$ is symmetric.
\end{prop}
\begin{proof}
From the commutation relations between the entries of the matrix
$\bff$ it follows that 
\begin{equation*}
\begin{aligned}
(\bff^2)_{ij}-(\bff^2)_{ji} &= \sum_l\bff_{il}\bff_{lj}-\bff_{jl}\bff_{il}=
\sum_l [\bff_{il},\bff_{lj}] = \\ 
& = \sum_l (\bff_{ij}-\del_{ij}\bff_{ll}-\del_{il}\bff_{lj}+\del_{jl}\bff_{li})
=(N-2)\bff_{ij},
\end{aligned}
\end{equation*}
and this is equivalent to the statement we needed to prove.
\end{proof}

\begin{theorem}\label{thm:transfer_spo}
Let  $(\gg, \gp)=(\oN, \spk)$ with $4k\leq N$ and let $I'$ be the
annihilator of the trivial representation of $\spk$. Then its
transfer $I$ is a quantization of the closure of the small nilpotent
orbit $\OO_{2k}$ of rank $2k$ in $\oN$. Moreover,
\begin{enumerate}
\item[(A)]
The matrix 
$p(\bff)$, where $p(u)=(u-k)(u-(N/2-k-1))$, is symmetric, its entries
span an $\ad$-invariant subspace of $\uoN$, and together with 
the quantum pfaffians of order $k+1$, they generate the ideal $I$;
\item[(B)]
If $N=2n$ is even and $2k<n-1$ then the corresponding primitive quotient is
isomorphic to the ring of the algebraic differential operators 
on the variety of skew-symmetric $n\times n$ matrices of rank at most $2k$.
\end{enumerate}
\end{theorem}
\begin{proof}
We need to determine the result of the $\spk$-invariant 
multiplication pairing of the subspace
of $\weyl$ spanned by the entries of $L(\bff')$ with the first summand
in the decomposition \eqref{eq:decomp_spo}. First, let us transform 
$(\rmp^{t}(\rmpst)\rmpt)_{ai} = \sum_b(\rmpt\rmpst)_{ab}\rmp_{ib}$:
\begin{equation*}
\sum_{b,l}\rmp_{la}\rmps_{bl}\rmp_{ib}=
\sum_{b,l}[\rmp_{la},\rmps_{bl}]\rmp_{ib}+\rmps_{bl}[\rmp_{la},\rmp_{ib}] +
[\rmps_{bl},\rmp_{ib}]\rmp_{la}+\rmp_{ib}\rmps_{bl}\rmp_{la}.
\end{equation*}
Here and below the indices $a,b$ corresond to the symplectic vector 
space $U$ and run from $1$ to $2k$, and the indices $i,j,l$ correspond
to the orthogonal vector space $V$ and run from $1$ to $N$.
Substituting $[\rmp_{la},\rmps_{bl}] = -\del_{ab}, 
[\rmp_{la}, \rmp_{ib}] = \del_{il}\rmj_{ab}$ and taking into account
that $\sum_b\rmj_{ab}\rmps_{bi}=\rmp_{ia}$, we get 
\begin{equation*}
(\rmp^{t}(\rmpst)\rmpt)_{ai} = 
-N\rmp_{ia} + \rmp_{ia} + 
2k\rmp_{ia} + \sum_l(\rmp\rmps)_{il}\rmp_{la} = 
(-N+2k+1)\rmp_{ia}+(\rmp\rmps\rmp)_{ia}.
\end{equation*}
Thus we have established the following \emph{convolution formula}:
\begin{equation}
(\rmp^{t}(\rmpst)\rmpt)_{ai} = 
\sum_l(\rmp\rmps+(-N+2k+1)\rmi_N)_{il}\rmp_{la} =
((\rmp\rmps\rmp+(-N+2k+1)\rmp)_{ia}.
\end{equation}
Now multiply both sides by $\rmps_{aj}$ and sum over $a$ to get
\begin{equation*}
\sum_{ab}(\rmpt\rmpst)_{ab}\rmp_{ib}\rmps_{aj} = 
(\rmp\rmps)^{2}+ (-N+2k+1)\rmp\rmps.
\end{equation*}
Recalling the formulas \eqref{eq:gen_spo} for the images of $\bff'$ and
$\bff$ in the Weyl algebra, we thus obtain
\begin{equation}\label{eq:transfer_spo}
\sum_{ab}L(\bff')_{ab}\rmp_{ib}\rmps_{aj} = 
R(p(\bff))_{ij}, \quad p(u)=(u-k)(u-(\frac{N}{2}-k-1)).
\end{equation}
This completes the computation of the $\spk$-invariant multiplication pairing.
The subspace of $\weyl$ spanned by the right hand side of 
\eqref{eq:transfer_spo} is contained in $\I=(I'\cdot\weyl)\cap R(\ug)$ and
generates it as a $R(\ug)$-module. Indeed, 
its associated graded is the subspace $U_\sig$ from Theorem 
\ref{thm:generate} that generates $\gr\I$. We conclude that the entries
of $p(\bff)$ generate $I=\theta(I')$ modulo the orthogonal rank ideal
$\ker R$. Therefore, the quadratic elements $p(\bff)_{ij}$ and 
the quantum pfaffians of order $k+1$ together generate 
the ideal $I$.

As in the general linear case, Part (B) follows from \cite{LS}, where
the ring of differential operators on the singular rank variety was
identified with the primitive quotient of $\uoN$ by the annihilator 
of the theta lift of the one-dimensional representation of the compact 
form of the symplectic group.
\end{proof}
%
%
%
\section{Relation between transfer and the Howe duality}\label{sec:duality}
%
%
In the previous section we have given applications of the transfer map
to explicit quantization of the orbits in classical Lie algebras. However,
our original motivation for introducing this map came from the theory of 
reductive dual pairs and the local theta correspondence over $\Real$. In
short, the transfer $\theta(I')$ of an ideal $I'$ in $\ugp$ provides a 
lower bound on the annihilator ideal $I$ in $\ug$ of representations in the
Howe duality correspondence with any representation with annihilator 
ideal $I'$.
We consider two algebraic versions of the Howe duality 
and use them to prove the lower bound theorem. Then we review the approach
to Capelli identities based on the theory of reductive dual pairs, introduce
noncommutative Capelli identities and explain 
how the transfer map
may be viewed as a system of identities determined
by a reductive dual pair $(\gg, \gp)$ and an ideal $I'$ of $\ugp$.
(A different approach to noncommutative
Capelli identities in the general linear case, not involving dual pairs, 
was developed in \cite{Okounkov}.)
%
%
\subsection{Proof of the lower bound on the annihilator ideal}
%
%
Our proof of Theorem \ref{thm:bound} establishing a lower bound on the
annihilator of a representation arising from Howe correspondence
relies on the description of the Howe duality for admissible representations
in the real case given in \cite{Howe_Transcending},
which we now recall. Suppose that $(G_0,G'_0)$ is a
reductive dual pair of real Lie groups in the symplectic group 
$Sp=Sp(W_0)$, where $W_0$ is a symplectic real vector space. 
We denote by $\gg$ and $\gp$ the complexified
Lie algebras of $G_0$ and $G'_0$, which form a reductive dual pair of Lie 
algebras. 
The metaplectic two-fold cover $\widetilde{Sp}$ of $Sp$ acts
in the space $\omega$ of the oscillator representation and this action
induces linear representations of the metaplectic two-fold covers 
$\widetilde{G_0}$ 
and $\widetilde{G'_0}$ of the members of the dual pair. 
One may choose a maximal compact subgroup $U$ of $Sp$ 
($U$ is a unitary group) in such a way that the associated Cartan 
decomposition of $Sp$ induces Cartan decompositions of $G_0$ and $G'_0$.
Thus it makes sense to consider restrictions of a Harish-Chandra module
of $Sp$ to $G_0$ and $G'_0$, and likewise for the metaplectic covers.
From now on we replace
representations of groups with the corresponding $(\gg,K)$-modules
of $K$-finite vectors. The complexified Weyl algebra $\weyl$ 
is identified with 
the algebra of the endomorphisms of $\omega$ whose $\widetilde{Sp}$-conjugates
span a finite-dimensional subspace of $\End\omega$.
\begin{definition}
Irreducible admissible
representations $\rho$ and $\rho'$ of $\widetilde{G_0}$ 
and $\widetilde{G'_0}$ correspond to each other under the Howe duality if
there exists an 
$\Ad(\widetilde{G_0}\cdot\widetilde{G'_0})$-invariant subspace $\omega_0$
of $\omega$ such that the quotient module 
$\omega/\omega_0$ is isomorphic to $\rho\otimes\rho'$. 
\end{definition}
The main theorem of 
\cite{Howe_Transcending} asserts that this correspondence
is a partial bijection between the sets
of equivalence classes of irreducible admissible representations of
the metaplectic covers of the members of the reductive dual pair.
``Forgetting'' about the actions of the maximal compact subgroups on $\rho$ 
and $\rho'$, we can formulate an entirely algebraic version of 
Howe duality for representations.
\begin{definition}\label{def:alg_duality}
Let $(\gg, \gp)$ is a reductive dual pair of Lie algebras and
$\A$ be a fixed non-zero module over the corresponding Weyl algebra $\weyl$. 
Then a $\gg$-module $V$ and a $\gp$-module $V'$  
are in the algebraic Howe duality with each other if there exists an 
$\Ad(G)$-invariant and $\Ad(G')$-invariant subspace $\A_0$ of $\A$ such
that the quotient $\A/\A_0$ is isomorphic to $V\otimes V'$ as a $\gg$-module
and a $\gp$-module.
\end{definition}
\begin{remark}\label{rem:detect}
The annihilator ideals of the modules $V$ and $V'$ in $\ug$ and $\ugp$ 
can be detected from the quotient 
$\A/\A_0$. They are the annihilators of this quotient viewed as a $\ug$-module
and as a $\ugp$-module.
\end{remark}  
The transfer map $\theta$ comprises of two steps. 
The first step takes the ideal $I'$ of $\ugp$
into the $\Ad(G\cdot G')$-invariant right ideal $\I=L(I')\cdot\weyl$ 
of $\weyl$, and the
second step replaces $\I$ with its pullback to $\ug$, $R^{-1}(\I)=\theta(I')$.
This observation forms a basis for our final algebraic approximation to
Howe duality.
\begin{definition}\label{def:idealsduality}
Let $(\gg,\gp)$ be a reductive dual pair of Lie algebras. Then two-sided
ideals $I\subset\ug$ and $I'\subset\ugp$ are in the algebraic 
Howe duality if there exists an $\Ad(G)$ and $\Ad(G')$-invariant 
right ideal $\I$ of the Weyl algebra such that $I$ and $I'$ 
are its pullbacks with respect to the maps $R$ and $L$ into $\weyl$.
\end{definition}
It is clear that if ideals $I'$ and $I$ are in algebraic duality then 
$\theta(I')\subseteq I$, for
\begin{equation*}
\theta(I')=R^{-1}(L(I')\weyl)\subseteq R^{-1}(\I)=I.
\end{equation*}
Moreover, if two admissible representations of the metaplectic covers 
of $G_0$ and $G'_0$ correspond to each other under the ordinary Howe 
duality then their underlying
$\gg$-module and $\gp$-module are in the algebraic Howe duality in the 
sense of Definition \ref{def:alg_duality}. 
Therefore, Theorem \ref{thm:bound} follows from the following
more general statement.
\begin{theorem}\label{thm:algtransfer}
Let $V$ and $V'$ be two modules in the algebraic Howe duality in the sense
of Definition \ref{def:alg_duality}. Denote the annihilator of $V$ in $\ug$ 
by $I$ and the annihilator of $V'$ in $\ugp$ by $I'$. Then 
the ideals $I$ and $I'$ are in the algebraic Howe duality. 
\end{theorem}
\begin{proof}
Let $\I$ be the annihilator in $\weyl$ of $\A/\A_0$.
Then $\I$ is an  
$\Ad(G)$-invariant and $\Ad(G')$-invariant right ideal 
of the Weyl algebra whose pullbacks to $\ugp$ and $\ug$ via the maps
$L$ and $R$ are precisely $I'$ and $I$, establishing the claim.
\end{proof}
There is a weak converse to Theorem \ref{thm:algtransfer}, 
\cf Remark \ref{rem:detect}.
\begin{theorem}
Let $I$ and $I'$ be two ideals in the algebraic Howe duality 
in the sense of Definition \ref{def:idealsduality}. Then there exists
a $\weyl$-module $\A$ and an $\Ad(G)$-invariant and $\Ad(G')$-invariant
subspace $\A_0$, 
such that $\A/\A_0$ has annihilator $I$ as a $\ug$-module 
and annihilator $I'$ as a $\ugp$-module.
\end{theorem}
\begin{proof}
Suppose that $\I$ be the right ideal of $\weyl$ realizing the algebraic
Howe duality between $I$ and $I'$. Let $\A=\weyl$ and $\A_0=\I$.
Then the conditions that $I=R^{-1}(\I)$ and $I'=L^{-1}(\I)$
translate into the conditions that the annihilators of $\A/\A_0$
as a $\ug$-module and as a $\ugp$-module are $I$ and $I'$.  
\end{proof}
%
%
%
\subsection{Noncommutative Capelli identities}
%
%
The complex general linear groups $GL_k$ and $GL_n$ act on the vector space 
$\mkn$ of complex $k\times n$ matrices by matrix left and right 
multiplication. Thus one obtains a pair of commuting actions of their 
Lie algebras by polynomial coefficient differential operators. 
The generators of $\glk$ and $\gln$ act on polynomials 
via polarization operators of the classical 
invariant theory. Capelli's remarkable accomplishment was to prove
the full isotypic decomposition of the space of polynomial functions 
on $\mkn$ under the pair
of commuting actions of $GL_k$ and $GL_n$, using the Capelli identity
as a tool. This 
\emph{Gordan--Capelli decomposition}, also known as the $(GL_n, GL_k)$-duality,
is multiplicity-free, and the representations of the two groups that occur 
in it determine each other.
In \cite{Howe_Umeda}, Howe and Umeda elucidated the role of 
the Capelli identity in the proof. The two general linear Lie algebras
involved form a reductive dual pair $(\gln, \glk)$ and this leads to
a relation between the centers of their universal enveloping algebras. 
Specifically, let us assume that $k\leq n$, then the Capelli identity 
provides an isomorphism between $\zgk$ and a certain
subalgebra/quotient of $\zgn$, both of which are polynomial algebras
with $k$ explicitly given generators. This isomorphism arises from the 
identity
\begin{equation}\label{eq:abs_Capelli}
R(C_i)=L(C'_i)
\end{equation}
in the algebra of polynomial coefficient
differential operators on $\mkn$
and is determined by the formula $C'_i\leftrightarrow C_i$ for $1\leq i\leq k$.
Here $C'_i$ and $C_i$ are
the $i$th (central) Capelli elements of degree $i$ in $\ugk$ and $\ugn$, 
$1\leq i\leq n$, and $C'_i=0$ for $i\geq k+1$. 
Setting $C_i=0$ for $i\geq k+1$ exhibits the algebra $\mk[C_1,\ldots, C_k]$
as a quotient of $\zgn$. By the theory of highest weight, finite-dimensional 
simple $\gln$-modules and $\glk$-modules are uniquely determined by their
infinitesimal character, therefore, the Capelli identity \eqref{eq:abs_Capelli}
determines the matching of $\gln$-modules and $\glk$-modules 
in the Gordan-Capelli decomposition.

The actions of the groups $GL_n$ and $GL_k$ by matrix multiplications
are complexifications of the 
actions of their maximal compact subgroups, the compact 
unitary groups $U_n$ and $U_k$ forming
a reductive dual pair $(U_n, U_k)$ in the real symplectic
group $Sp_{2nk}(\Real)$. 
Up to twists by certain powers of the determinant characters, 
these latter actions arise by restricting 
the oscillator representation of the symplectic group, realized 
in the Fock model, to the members of the compact dual pair. (The twists
are reflected in the difference between the normalized and unnormalized 
versions of the maps $L$ and $R$.) In this setting, the
Capelli identity \eqref{eq:abs_Capelli} between the images of 
the algebra generators
of the centers of $\ug$ and $\ugp$ completely describes the matching of 
representations under the Howe duality. 
Now, let us consider a reductive dual pair 
$(U_{p,q}, U_{r,s})$ of definite or indefinite unitary groups in 
$Sp_{2(p+q)(r+s)}(\Real)$, and assume that at least one of the groups is
non-compact.
The classical Capelli identity leads to a correspondence of the 
infinitesimal characters for those representations $\rho$ and $\rho'$ 
of the metaplectic covers of $G_0$ and $G'_0$ which correspond to each other
under the Howe duality. However, there is a fundamental difference 
between  representation theory of compact and noncompact reductive groups: 
in the noncompact case, irreducible admissible representations are
not determined by the infinitesimal character alone. Thus in this case
the Capelli identity yields only partial information on the Howe duality
correspondence. Capelli identities for other reductive dual pairs, which
were established by Molev--Nazarov and Itoh, 
\cite{Itoh_Advances, Molev_Nazarov}, involve only the central elements,
consequently, they suffer from the same defect.

The transfer map $I'\mapsto \theta(I')$ serves to sharpen the description
of the Howe duality, by incorporating more refined information on the 
representations $\rho$ and $\rho'$, namely, their annihilator (primitive)
ideals. 
\begin{definition}
Given a family of representations of $\gg$, 
a \emph{noncommutative Capelli identity} is an element of the 
universal enveloping
algebra $\ug$ which acts by $0$ on any representation $\rho$ from this
family. An identity is \emph{central} if the element belongs to $\zg$.
A set of noncommutative Capelli identities 
is called a \emph{system of noncommutative Capelli identities}
if it spans an $\ad$-invariant subspace of $\ug$.
\end{definition}
Noncommutative Capelli identities form an ideal
in $\ugn$ and that a (minimal) $\ad$-invariant
generating subspace produces a (minimal) system of identities, 
with others being their formal consequences. Next we present some
examples.

\begin{enumerate}
\item
The theory of rank for ideals of $\ugn$, \cite{Protsak_RankCapelli}, gave
the first examples of systems of noncommutative Capelli identities
that are not central. In this special case, $I'$ is the zero ideal, 
$\theta(I')=\ker R$ 
is the $k$th rank ideal of $\ugn$ and is generated by the
(normalized) quantum minors of order $k+1$. The corresponding identities 
are satisfied in any representation of the metaplectic cover of 
$U_{p,q}$ which appears
in the Howe duality with a representation of the metaplectic cover of
$U_{r,s}$, where $p+q=n, r+s=k$. 
\item
More generally, 
Theorem \ref{thm:bound} provides a system  $\theta(I')$
of noncommutative Capelli identities for representations of $\gg$ that 
occur in the Howe duality correspondence with a representation of $\gp$ whose
annihilator $I'$ has been fixed.
\item
As a specialization, Corollary \ref{cor:transfer_gl} 
(in the general linear case) and 
Theorem \ref{thm:transfer_spo} (in the symplectic--orthogonal case)
described an $\ad$-invariant generating space
for the transfer of a codimension one ideal. It follows from 
Theorem \ref{thm:bound} that we have thus obtained a system of 
noncommutative Capelli identities
satisfied on the theta lifts of one-dimensional representations.
The degree $2$ elements are new. 
\item
Further examples are obtained from the theory of minimal polynomials
and quantized elementary divisors for $\gg$-modules, 
\cite{Protsak_minpol, Protsak_divisor}.
\end{enumerate}
%
%
\section{Complements and open questions}
%
%
%
\subsection{Skew analogue}
%
The theory developed in this paper has a natural ``skew'' analogue for 
reductive dual pairs in an orthogonal Lie algebra, with Clifford
algebra in place of Weyl algebra,  
\cf \cite{Howe_Remarks, Howe_Schur, Itoh_Advances}. 
Since unlike the oscillator representation, the spin representation is 
finite-dimensional and decomposes into a direct sum of the tensor products of
simple modules for $G$ and $G'$, 
in effect one is dealing with finite-dimensional
modules and their annihilators, which completely determine each other
when $G$ and $G'$ are connected (i.e. excluding the orthogonal -- orthogonal
case). This results in a streamlined behavior
of the transfer map, which becomes a bijection between finite
sets consisting of the annihilators of certain
finite-dimensional simple modules 
for $G$ and $G'$. On the other hand, in the skew-symmetric case the 
moment maps and quantization are no longer relevant to the story.
%
%
\subsection{Effect on primitive ideals}
%
%
From our algebraic perspective, the most interesting question about
transfer concerns the behavior of primitive and completely prime ideals.
\begin{question}
Let $(\gg,\gp)$ be an irreducible reductive dual pair of Lie algebras
in the stable range
with the smaller member $\gp$. Suppose that $I'$ is a primitive ideal
of $\ugp$. Is it true that its transfer $\theta(I')$ is a primitive
ideal of $\ug$? The same question replacing ``primitive''
with ``completely prime''.
\end{question}
More generally, suppose that $(\gg,\gp)$ is a reductive dual pair with
$\gp$ the smaller member, then one might ask the same questions for
sufficiently small ideals $I'$ of $\ugp$. The precise condition is that
the associated variety of $I'$ has nonempty intersection with a 
Zariski open subset of $\gps$ over which the moment map $l$ is 
equidimensional with irreducible fibers. For the general linear case 
this amounts to the restriction
that the corresponding nilpotent orbit contains a matrix whose rank 
is at least $2k-n$. 
Partial progress towards answering these questions 
has been made in \cite{Protsak_minpol, Protsak_divisor}.

In the general linear case, the set of primitive ideals admits a combinatorial
para\-me\-tri\-za\-tion in terms of Young tableaux that is
essentially due to Joseph.
In \cite{Protsak_Thesis}, I used it to construct a bijection between 
the set of primitive ideals of $\ugk$ and the set of primitive ideals 
of $\ugn$ that have \emph{pure rank} $k$. I think that 
this combinatorially defined bijection  describes transfer
for the reductive dual pair $(\gln, \glk)$ with $n\geq 2k$ at
the level of Young tableaux, but I have not been able to prove it.
%
%
%
\subsection{Is the transfer bound exact?}
%
%
Let $(G_0,G'_0)$ be a reductive dual pair of real groups
in stable range, with $G'_0$ the smaller member, and suppose
that irreducible admissible representations $\rho$ and $\rho'$ 
of their metaplectic covers are in Howe duality.
It is natural to conjecture 
that if the representation $\rho'$ is unitary 
then $\theta(\ann\rho')=\ann\rho$, in other words, that the 
transfer bound is sharp. (I had first learned a more qualitative form of
this conjecture from P.~Trapa.)
J.-S.~Li's description of singular unitary representations 
and T.~Przebinda's results on the behavior of the wave front sets
under the Howe duality, \cite{Li_Singular, Przebinda_Unitary}, 
constitute strong supporting evidence for the conjecture.
In particular, in the stable range, the annihilator of the theta-lift of
a one-dimensional representation of the smaller group 
is generated by the elements from Corollary 
\ref{cor:transfer_gl} and Theorem \ref{thm:transfer_spo}.
Nonetheless, explicit computations of the Howe duality correspondence
indicate that neither the stable range condition nor the unitarity 
of $\rho'$ could be completely eliminated from the assumptions.
\subsection*{Acknowledgment} 
This article is based on research carried 
out at Yale University, MPI f\"ur Mathematik, Bonn, University of Minnesota,
and Cornell University. I express my sincere
gratitude to all these institutions. 
%
%
%
\bibliographystyle{plain}
\bibliography{refer} 
%

\end{document}